\documentclass[11pt]{article}
\usepackage{bbm}
\usepackage{amssymb}
\usepackage{amsfonts}
\usepackage{mathrsfs}
\usepackage{amssymb,amsmath,amsthm,amsfonts}
\usepackage{calligra}
\usepackage[T1]{fontenc}
\usepackage{color}
\usepackage{graphicx}
\newtheorem{theorem}{Theorem}                  
\newtheorem{prop}{Proposition}
\newtheorem{lemma}{Lemma}

\begin{document}
\title{\LARGE   \textbf{The edit distance function of some graphs } }

\author{
\Large \textbf{Yumei Hu$^{a}$, Yongtang Shi$^{b}$, Yarong Wei$^{a,*}$ }  \\
 \emph{$^{a}$ School of Mathematics}\\
 \emph{Tianjin University, Tianjin 300072, China}\\
 \emph{\texttt{huyumei@tju.edu.cn; yarongwei@163.com}}\\
 \emph{$^{b}$ Center of Combinatorics and LPMC}\\
 \emph{Nankai University, Tianjin 300071, China}\\
 \emph{\texttt{shi@nankai.edu.cn}}\\
 }

\date{}
\maketitle

\begin{abstract}
The edit distance function of a hereditary property $\mathscr{H}$ is the asymptotically largest edit distance  between a graph of density $p\in[0,1]$ and $\mathscr{H}$. Denote by $P_n$ and $C_n$ the path graph of order $n$ and the cycle graph of order $n$, respectively. Let $C_{2n}^*$ be the cycle graph $C_{2n}$ with a diagonal, and $\widetilde{C_n}$ be the graph with vertex set $\{v_0, v_1, \ldots, v_{n-1}\}$ and $E(\widetilde{C_n})=E(C_n)\cup \{v_0v_2\}$.
Marchant and Thomason determined the edit distance function of $C_6^{*}$.
Peck studied the edit distance function of $C_n$, while Berikkyzy et al. studied the edit distance of powers of cycles. In this paper, by using the methods of Peck and Martin, we determine the edit distance function of $C_8^{*}$, $\widetilde{C_n}$ and $P_n$, respectively.

\end{abstract}



\section{Introduction}
The edit distance in graphs was introduced by Axenovich, K\'ezdy and Martin \cite{r1} and by Alon and Stav \cite{r2} independently.
The edit distance problem considered here is ``How many edges need to be added or deleted (edited) in a graph $G$ so that it will have a certain property?'' The presence or absence of edges in a certain graph corresponds to pairs of genes which activate or deactivate one another in evolutionary biology.
In evolutionary theory, the gene reconstruction avoiding forbidden induced subgraphs is studied \cite{r3}, which is equivalent to the edit distance problem. The edit distance problem is also important to the algorithmic aspects of property testing \cite{r4,r6,r7,r2}.

The \emph{edit distance} between a graph $G$ and a property $\mathscr{H}$ is
$$dist(G, \mathscr{H})=\min\left\{|E(G)\bigtriangleup E(G')|/{n \choose 2}:V(G)=V(G'), G'\in\mathscr{H}\right\}.$$
The \emph{edit distance function} of a property $\mathscr{H}$, denoted $ed_{\mathscr{H}}(p)$, measures the maximum distance of a graph with density $p$ from $\mathscr{H}$. Formally,
\begin{equation}\label{e3}
ed_{\mathscr{H}}(p)=\lim_{n\rightarrow\infty}\max\left\{dist(G, \mathscr{H}): |V(G)|=n, |E(G)|=\left\lfloor p{n \choose 2}\right\rfloor\right\}.
\end{equation}
if this limit exists.

A \emph{hereditary property} is a family of graphs that is closed under the taking of induced subgraphs. For a given graph $H$, the property of having no $H$ as an induced subgraph  is called a \emph{principal hereditary property}, denoted by $Forb(H)$. Clearly, $Forb(H)$ is a hereditary property for any graph $H$. In fact, for every hereditary property $\mathscr{H}$ there exists
a family of graphs $\mathscr{F}(\mathscr{H})$ such that $\mathscr{H}=\bigcap_{H\in\mathscr{F}(\mathscr{H})}Forb(H)$.
A hereditary property is said to be \emph{nontrivial} if there is an infinite sequence of graphs that is in the property.
The properties for which we study the edit distance are usually hereditary property.

Balogh and Martin \cite{r51} showed that the limit in (\ref{e3}) exists and the edit distance function has a number of interesting properties:

\begin{prop}[\cite{r24}]\label{pro8}
Let $\mathscr{H}$ be a nontrivial hereditary property. For $p\in[0, 1]$,

 $(a)$ $ed_\mathscr{H}(p)$ is continuous.

 $(b)$ $ed_\mathscr{H}(p)$ is concave down.
\end{prop}

In \cite{r2}, Alon and Stav prove that for every hereditary property $\mathscr{H}$, there exists a  $p^{*}=p^{*}(\mathscr{H})\in[0, 1]$ such that the maximum distance of a graph $G$ on $n$ vertices from $\mathscr{H}$ is asymptotically the same as that of the Erd\"os-R\'enyi random graph $G(n,p^{*})$. Namely,
\begin{equation}\label{e5}
\max\left\{dist(G, \mathscr{H}): |V(G)|=n\right\}=\mathbb{E}[dist(G(n, p^{*}), \mathscr{H})]+o(1).
\end{equation}
We denote the limit in (\ref{e5}) by $d^{*}_{\mathscr{H}}$.

The edit distance functions of some kinds of graphs have been investigated in recent years, including complete graphs \cite{r13} and split graphs \cite{r9}. Actually, complete bipartite graphs are also studied. Marchant and Thomason \cite{r32} studied the edit distance functions of $K_{2,2}$ and $K_{3,3}$, respectively. Balogh and Martin \cite{r51} established the value of $p^{*}_{Forb(K_{3,3})}$ and $d^{*}_{Forb(K_{3,3})}$.
Martin and McKay studied the edit distance function of $K_{2,t}$ in \cite{r14}. Recently, Berikkyzy et al. \cite{r5} settled the edit distance function for many powers of cycles.

Denote by $P_n$ and $C_n$ the path graph of order $n$ and the cycle graph of order $n$, respectively. Let $C_{2n}^*$ be the cycle graph $C_{2n}$ with a diagonal, and $\widetilde{C_n}$ be the graph with vertex set $\{v_0, v_1, \ldots, v_{n-1}\}$ and $E(\widetilde{C_n})=E(C_n)\cup \{v_0v_2\}$.

In \cite{r32}, Marchant and Thomason studied the edit distance function of the graph $C_6^{*}$. Motivated by this result, we study the edit distance function of the graph $C_8^{*}$ and prove the following result.

\begin{theorem}\label{theorem1}
Let $\mathscr{H}=Forb(C_8^*)$.

$$ed_{\mathscr{H}}(p)= \min\left\{\frac{p}{2}, \frac{p(1-p)}{1+p}, \frac{1-p}{3}\right\},\ \text{for}\ p\in[0, 1].$$
\end{theorem}

Peck \cite{r8} in her Master's thesis calculated the edit distance function of $C_n$. The result is as follows,

\begin{theorem}[\cite{r8}]\label{theorem6}
Let $\mathscr{H}=Forb(C_n)$.

$(a)$ If $n$ is odd, then $ed_{\mathscr{H}}(p)= \min\left\{\frac{p}{2}, \frac{p(1-p)}{1+(\lceil\frac{n}{3}\rceil-2)p}, \frac{1-p}{\lceil\frac{n}{2}\rceil-1}\right\}$, \text{for} $p\in[0,1]$.

$(a)$ If $n$ is even, then $ed_{\mathscr{H}}(p)= \min\left\{ \frac{p(1-p)}{1+(\lceil\frac{n}{3}\rceil-2)p}, \frac{1-p}{\lceil\frac{n}{2}\rceil-1}\right\}$, for $p\in\left[\ \left\lceil n/3\right\rceil^{-1},1\right]$.
\end{theorem}

Motivated by this result, we study the edit distance function of $\widetilde{C_n}$ and $P_n$.

\begin{theorem}\label{theorem5}
Let $\mathscr{H}=Forb(\widetilde{C_n})$ and $n\geq 9$.

$$ed_{\mathscr{H}}(p)= \min\left\{\frac{p}{2}, \frac{p(1-p)}{1+(\lceil\frac{n-1}{3}\rceil-2)p}, \frac{1-p}{\lceil\frac{n-3}{2}\rceil}\right\},\ \text{for}\ p\in[0, 1].$$
\end{theorem}

\begin{theorem}\label{theorem2}
Let $\mathscr{H}=Forb(P_n)$ and $n\geq 3$.

$$ed_{\mathscr{H}}(p)= \min\left\{ \frac{p(1-p)}{1+(\lceil\frac{n-1}{3}\rceil-2)p}, \frac{1-p}{\lceil\frac{n}{2}\rceil-1}\right\},\ \text{for}\ p\in\left[\ \left\lceil (n-1)/3\right\rceil ^{-1}, 1\right].$$
\end{theorem}

Our paper is organized as follows. Some definitions and tools are explained in Section \ref{Sec2}. We prove Theorems \ref{theorem1}, \ref{theorem5} and \ref{theorem2}, in Sections \ref{Sec3}, \ref{Sec5} and \ref{Sec4}, respectively.

\section{Definitions and Tools}\label{Sec2}
All graphs considered in this paper are simple. The standard graph theory notation not defined here will conform to that in \cite{gtwa}. The edit distance notation not defined here will conform to that in \cite{r24}.

In order to estimate the edit distance function, Alon and Stav \cite{r2} defined a \emph{colored regularity graph} (CRG) $K$ as follows. Let $K$ be a simple
complete graph, together with a partition of the vertices into white and black,
and a partition of the edges into white, gray, and black. Denote by $VW(K)$ and $VB(K)$ the set of white vertices and the set of black vertices, respectively. Then $V(K)=VW(K)\cup VB(K)$. Denote by $EW(K)$, $EG(K)$ and $EB(K)$ the set of white edges, the set of gray edges and the set of black edges, respectively. Then we have $E(K)=EW(K)\cup EG(K)\cup EB(K)$. A CRG $K'$ is said to be a \emph{sub}-CRG of $K$ if $K'$ can be obtained by deleting vertices of $K$ and is a \emph{proper sub}-CRG if $K'\not=K$.

We say that a graph $H$ \emph{embeds in} $K$ (writing $H\mapsto K$), if there is a function $\varphi: V(H)\rightarrow V(K)$ so that if $h_1h_2\in E(H)$, then either $\varphi(h_1)=\varphi(h_2)\in VB(K)$ or $\varphi(h_1)\varphi(h_2)\in EB(K)\cup EG(K)$ and if $h_1h_2\not\in E(H)$, then either $\varphi(h_1)=\varphi(h_2)\in VW(K)$ or $\varphi(h_1)\varphi(h_2)\in EW(K)\cup EG(K)$.
For a hereditary property $\mathscr{H}$, we denote $\mathscr{K}(\mathscr{H})$ to be the subset of CRGs $K$ such that any graph $H\in \mathscr{F}(\mathscr{H})$ does not embed in $K$. That is,  $\mathscr{K}(\mathscr{H})=\{K: H\not\mapsto K, \forall H\in \mathscr{F}(\mathscr{H})\}$.

For a hereditary property $\mathscr{H}$, we can use the $g$ function of each CRG $K$ to compute the edit distance function, where $g$ function is defined by

\begin{equation}\label{e2}
g_K(p)= \min \left\{\mathbf{x}^{T}M_K(p)\mathbf{x}:\mathbf{x}^{T}\mathbf{1}=1,\mathbf{x} \mathbf{\geq 0}\right\},
\end{equation}
and
$$ [M_K(p)]_{ij}=\begin{cases}
p           & \text{if}\ v_iv_j\in EW(K)\ \text{or}\ v_i=v_j\in VW(K); \\
1-p         & \text{if}\ v_iv_j\in EB(K)\ \text{or}\ v_i=v_j\in VB(K);       \\
0           & \text{if}\ v_iv_j\in EG(K).
\end{cases}$$

Marchant and Thomason in \cite{r32} proved that for every $p\in[0,1]$, there is a CRG $K\in\mathscr{K}(\mathscr{H})$ such that $ed_\mathscr{H}(p)=g_K(p)$. That is

\begin{prop}[\cite{r32}]\label{pro7}
Let $\mathscr{H}$ be a nontrivial hereditary property. For $p\in[0, 1]$,
$$ed_\mathscr{H}(p)=\min \{g_K(p):K\in\mathscr{K}(\mathscr{H})\}.$$
\end{prop}

In \cite{r32}, the authors also proved that in order to find such CRGs, we only need to look at all $p$-core CRGs. A CRG $K$ is $p$-core if, for any proper sub-CRG $K'$ of $K$, we have $g_{K'}(p)>g_{K}(p)$.

The \emph{gray-edge} CRG $K(r, s)$ is the CRG $K$ with $r$ white vertices, $s$ black vertices and all edges gray. The \emph{clique spectrum} of $\mathscr{H}$ is the set $\Gamma(\mathscr{H}):=\{(r, s): H\not\mapsto K(r, s),\ \forall H\in \mathscr{F}(\mathscr{H})\}$. Clearly, we obtain

\begin{prop}[\cite{r24}]
Let $\mathscr{H}$ be a nontrivial hereditary property and $\Gamma(\mathscr{H})$ denote the clique spectrum of $\mathscr{H}$. If we define
$$\gamma_{\mathscr{H}}(p):= \min_{(r,s)\in\Gamma(\mathscr{H})}g_{K(r,s)}(p)=\min_{(r,s)\in\Gamma(\mathscr{H})}\frac{p(1-p)}{r(1-p)+sp},$$
then $ed_{\mathscr{H}}(p)\leq\gamma_{\mathscr{H}}(p)$.
\end{prop}

Let $K$ be a $p$-core CRG, $v\in V(K)$, and let $\mathbf{x}$ be an optimal weight vector in the quadratic program (\ref{e2}) that defines $g_K(p)$. The \emph{weight} of $v$, denoted by $\mathbf{x}(v)$, is the entry corresponding to $v$ of the vector $\mathbf{x}$. We denote the \emph{gray neighborhood} of $v$ by $N_G(v)=\left\{v{'}\in V(K): vv{'}\in EG(K)\right\}$. The \emph{weighted gray degree} of vertex $v\in V(K)$ is $d_G(v)=\Sigma_{v{'}\in N_G(v)}\mathbf{x}(v{'})$ and the number of vertices adjacent to $v$ via gray edges is denoted by $deg_G(v)$, i.e., $deg_G(v)=|N_G(v)|$. We use similar notation for the white and black cases. Now we get $d_G(v)+d_W(v)+d_B(v)=1$ for each $v\in V(K)$.

The \emph{weighted gray codegree} of vertices $v$ and $v'$, denoted by $d_G(v, v')$, is the sum of the weights of the common gray neighbors of $v$ and $v'$. Denote the number of common gray neighbors of vertices $v$ and $v'$ by $deg_G(v, v')$.

Marchant and Thomason {\cite{r32}} gave the following characterization of all $p$-core CRGs.

\begin{prop}[\cite{r32}]\label{pro1}
Let $K$ be a $p$-core CRG.

$(a)$ If $p\leq1/2$, then there are no black edges, and the white edges are only incident to black vertices.

$(b)$ If $p\geq1/2$, then there are no white edges, and the black edges are only incident to white vertices.
\end{prop}

Martin \cite{r13} gave a formula for $d_G(v)$ for all $v\in V(K)$ and a bound on the weight of each $v$.

\begin{prop}[\cite{r13}]\label{pro2}
Let $p\in(0,1)$ and $K$ be a p-core CRG with optimum weight vector $\mathbf{x}$.

$(a)$ If $p\leq1/2$, then $\mathbf{x}(v)=g_K(p)/p$ for all $v\in VW(K)$ and $$\mathbf{x}(u)\leq g_K(p)/(1-p),\ d_G(u)=\frac{p-g_K(p)}{p}+\frac{1-2p}{p}\mathbf{x}(u), \ \text{for}\ \text{each}\ u\in VB(K).$$

 $(b)$ If $p\geq1/2$, then $\mathbf{x}(u)=g_K(p)/(1-p)$ for all $u\in VB(K)$ and $$\mathbf{x}(v)\leq g_K(p)/p,\ d_G(v)=\frac{1-p-g_K(p)}{1-p}+\frac{2p-1}{1-p}\mathbf{x}(v),\ \text{for}\ \text{each}\ v\in VW(K).$$
\end{prop}

The following results will be used in this paper.

\begin{prop}[\cite{r13}]\label{pro3}
Let $p\in(0,1/2)$ and $K$ be a p-core CRG with black vertices and white or gray edges.

$(a)$ If $K$ has no gray 3-cycle, then $g_K(p)> p/2$.

 $(b)$ If $K$ has a gray 3-cycle, but no gray $C_4^+$(that is, four vertices that induce 5 gray edges), then $g_K(p)\geq min\left\{2p/3, (1-p)/3\right\}$.
\end{prop}

\begin{prop}[\cite{r5}]\label{pro5}
Let $F$ be a connected graph. If some path of maximum length forms a cycle, then $F$ is Hamiltonian.
\end{prop}

\begin{prop}[\cite{r5}]\label{pro6}
Let $F$ be a graph on $n$ vertices with no cycle of length longer than $\lceil\frac{n}{2}\rceil-1$, with every vertex having degree at least $\lceil\frac{n-1}{3}\rceil\geq 2$ and with every pair of vertices having at least one common neighbor. Furthermore, let $F$ have the property that no maximum length path forms a cycle.

Let $v_1 \ldots v_\ell$ be a path of maximum length in $F$. Then $v_1$ and $v_\ell$ have exactly one common neighbor $v_c$ on this path. Furthermore, $N(v_1)\subseteq\{v_2, \ldots, v_c\}$ and $N(v_\ell)\subseteq\{v_c, \ldots, v_{\ell-1}\}$.
\end{prop}

\section{Proof of Theorem \ref{theorem1}}\label{Sec3}

In this section, we consider the edit distance function for the hereditary property that forbids $C_{2n}^{*}$ where $n$ is even and prove that $ed_{Forb(C_8^*)}(p)=\gamma_{Forb(C_8^*)}(p)$ for all $p \in [0, 1]$.

First, we obtain the value of $\gamma_{Forb(C_{2n}^*)}(p)$ for $p\in [0,1]$ and restrict $ed_{Forb(C_{2n}^*)}(p)$ to $p\in [0, 1/2)$ and CRGs $K$ with only black vertices. Finally, we determine the edit distance $ed_{Forb(C_8^*)}(p)=\gamma_{Forb(C_8^*)}(p)$ and then prove Theorem \ref{theorem1}.

\begin{lemma}\label{lemma5}
Let $\mathscr{H}=Forb(C_{2n}^*)$, $p\in[0, 1]$ and $n\geq4$ is even.

$$\gamma_{\mathscr{H}}(p)= \min\left\{\frac{p}{2},\ \frac{p(1-p)}{1+(\lceil\frac{2n-1}{3}\rceil-2)p},\ \frac{1-p}{n-1}\right\}.$$
Furthermore, if there is a $p$-core CRG $K\in K(Forb(C_{2n}^*))$ such that $g_K(p) <\gamma_{Forb(C_{2n}^*)}(p)$ for any $p\in[0, 1]$, then $p < 1/2$ and K has all black vertices.
\end{lemma}

\begin{proof}
If $n$ is even, the extreme points of the clique spectrum of $Forb(C_{2n}^*)$ are $(2, 0)$, $\left(1,\left\lceil\frac{2n-1}{3}\right\rceil-1\right)$ and $(0,\ n-1)$. Then $\gamma_{\mathscr{H}}(p)= \min\left\{\frac{p}{2}, \frac{p(1-p)}{1+(\left\lceil\frac{2n-1}{3}\right\rceil-2)p}, \frac{1-p}{n-1}\right\}$.

Since $ed_\mathscr{H}(1/2)=\gamma_\mathscr{H}(1/2)$ for any hereditary property and $\gamma_\mathscr{H}(1)=0$, we may use continuity and concavity to conclude that $ed_\mathscr{H}(p)=\gamma_\mathscr{H}(p)=\frac{1-p}{n-1}$ for $p\in [1/2, 1]$.
Now we suppose $p\in [0, 1/2)$ and $K$ is a $p$-core CRG such that $g_K(p)<\gamma_\mathscr{H}(p)$.

If $K$ has only white vertices, then $|V(K)|\leq2$ and $g_K(p)\geq\frac{p}{2}\geq\gamma_\mathscr{H}(p)$ since $C_{2n}^{*}\mapsto K(3, 0)$. If $K$ has both white and black vertices, then it has 1 white vertex $\omega$ since $C_{2n}^*\mapsto K(2,1)$. Furthermore, it can have at most $ \left\lceil\frac{2n-1}{3}\right\rceil-1$ black vertices.

To see this, denote the vertices of $C_{2n}^*$ by $\{0,\ldots,2n-1\}$ where $i\sim i+1$ for $0\leq i\leq 2n-2$, $2n-1\sim 0$ and $0\sim n$. If $n$ is not divisible by 3, then let $S$ consist of the members of $\{0, 1, \ldots, 2n-1\}$ that are divisible by 3. The graph $C_{2n}^*-S$ has $\left\lceil\frac{2n-1}{3}\right\rceil$ connected components, each of which are cliques of size 1 or 2. If $n$ is divisible by 3, then let $S=\{i:i\in\{0, 1, \ldots, 2n-1\}, i-1$ is divisible by 3\}. The graph $C_{2n}^*-S$ has $\left\lceil\frac{2n-1}{3}\right\rceil$ connected components, each of which are cliques of size 2 except three edges $n-1\sim n$, $n\sim0$ and $0\sim 2n-1$.

If $d_G(v_i)=\mathbf{x}(\omega)$ for any $v_i\in VB(K)$, then by Proposition \ref{pro2}$(a)$, we have $\frac{g_K(p)}{p}=\frac{p-g_K(p)}{p}+\frac{1-2p}{p}\mathbf{x}(v_i)>\frac{p-g_K(p)}{p}$.
Rearranging the terms, we obtain $g_K(p)> \frac{p}{2}\geq \gamma_\mathscr{H}(p)$, a contradiction.
So, there are two black vertices $v_1, v_2$ in $K$ such that $v_1v_2\in EG(K)$. Let $v_1$ receive $n-1\sim n$ and $v_2$ receive $0\sim 2n-1$, then $C_{2n}^{*}\mapsto K$.
Thus, regardless of whether the edges are white or gray, there are at most $\lceil\frac{2n-1}{3}\rceil-1$ black vertices in $K$ and $g_K(p)\geq\frac{p(1-p)}{1+(\lceil\frac{2n-1}{3}\rceil-2)p}\geq\gamma_\mathscr{H}(p)$.

So, if $p\in [0, 1/2)$ and $g_K(p)=ed_{Forb(C_{2n}^{*})}(p)$, then $K$ is either $K(2, 0)$, $K(1,\left\lceil\frac{2n-1}{3}\right\rceil-1)$, $K(0, n-1)$ or $K$ has all black vertices (and white or gray edges).

\end{proof}

\begin{proof}[\textbf{Proof of Theorem \ref{theorem1}}]
Now, we calculate $ed_{\mathscr{H}}(p)$ where $\mathscr{H}=Forb(C_8^*)$. By Lemma \ref{lemma5}, we know $\gamma_{\mathscr{H}}(p)= \min\left\{\frac{p}{2}, \frac{1-p}{3}, \frac{p(1-p)}{1+p}\right\}$ and only need consider the $p$-core CRGs $K$ with only black vertices for some $p\in [0, 1/2)$.

If $K$ has only black vertices, then $K$ has no gray $C_4^+$ otherwise $C_8^*\mapsto K$. By Proposition {\ref{pro3}}, we know either $g_K(p)> p/2\geq \gamma_\mathscr{H}(p)$ or $g_K(p)\geq \min\left\{2p/3, (1-p)/3\right\}> \gamma_\mathscr{H}(p)$. By straightforward calculations, this contradicts to $g_K(p)<\gamma_\mathscr{H}(p)$ for all $p\in [0, 1/2)$.

\end{proof}

\section{Proof of Theorem \ref{theorem5}}\label{Sec5}
In this section, we consider the edit distance function for hereditary property that forbids $\widetilde{C_n}$. Let $\mathscr{H}=Forb(\widetilde{C_n})$. First, we obtain the value of $\gamma_{\mathscr{H}}(p)$ for $p\in [0,1]$. Then we suppose there is a $p$-core CRG $K\in \mathscr{K}(Forb(\widetilde{C_n}))$ such that $g_K(p)< \gamma_{\mathscr{H}}(p)$ and establish some characterizations of such a $p$-core CRG $K$. Finally, we obtain a contradiction to such a CRG existing in $\mathscr{K}(Forb(\widetilde{C_n}))$ for our desired range of $p$ values, establishing $\gamma_{\mathscr{H}}(p)\leq ed_{\mathscr{H}}(p)$.

\begin{lemma}\label{lemma1}
Let $\mathscr{H}=Forb(\widetilde{C_n})$, and $n\geq6$ then
$$\gamma_{\mathscr{H}}(p)= \min\left\{\frac{p}{2}, \ \frac{p(1-p)}{1+(\lceil\frac{n-1}{3}\rceil-2)p},\ \frac{1-p}{\left\lceil\frac{n-3}{2}\right\rceil}\right\},\ \text{for}\ p\in[0, 1].$$
Furthermore, if there is a $p$-core CRG $K\in\mathscr{K}(\mathscr{H})$ such that $g_K(p)<\gamma_\mathscr{H}(p)$ for any $p\in[0, 1]$, then $p < \frac{1}{2}$ and $K$ has all black vertices.
\end{lemma}

\begin{proof}
The extreme points of the clique spectrum of $Forb(\widetilde{C_n})$ are (2, 0), $\left(1, \left\lceil\frac{n-1}{3}\right\rceil-1\right)$ and $(0, \left\lceil\frac{n-3}{2}\right\rceil)$, and which establishes the value of $\gamma_\mathscr{H}(p)$.

Since $ed_\mathscr{H}(1/2)=\gamma_\mathscr{H}(1/2)$ for any hereditary property and $\gamma_\mathscr{H}(1)=0$, we may use continuity and concavity to conclude that $ed_\mathscr{H}(p)=\frac{1-p}{\left\lceil\frac{n-3}{2}\right\rceil}$ for $p\in [1/2, 1]$.

Now, let $p\in[0, 1/2)$ and $K$ be a $p$-core CRG such that $\widetilde{C_n}\not\mapsto K$. If $K$ has at most two vertices, then $g_K(p)\geq \frac{p}{2}$ since $\widetilde{C_n}\mapsto K(3, 0)$.
If $K$ has both white and black vertices, then it has at most one white vertex since $\widetilde{C_n} \mapsto K(2,1)$. Furthermore, it can have at most $\lceil\frac{n-1}{3}\rceil-1$ black vertices.

To see this, denote the vertices of $\widetilde{C_n}$ by $\{0,1,\ldots,n-1\}$ where $i\sim i+1$ for $0\leq i\leq n-2$, $n-1\sim0$ and $0\sim 2$. Let $S$ consist of the members of $\{3,\ldots,n-1\}$ that are divisible by 3. If $n-1$ is not divisible by 3, then add $0$ to $S$. The graph $\widetilde{C_n}-S$ has $\lceil\frac{n-1}{3}\rceil$ connected components, each of which are cliques of size 1 or 2 or 3. Thus, regardless of whether the edges are white or gray, there are at most $\lceil\frac{n-1}{3}\rceil-1$ black vertices in $K$ and $g_K(p)\geq\frac{p(1-p)}{1+(\lceil\frac{n-1}{3}\rceil-2)p}$, with equality if and only if $K\cong K(1, \lceil\frac{n-1}{3}\rceil-1)$.

Summarizing, if $p\in[0,1/2)$ and $g_K(p)=ed_\mathscr{H}(p)$, then $K$ is either $K(2,0)$, $K(1, \lceil\frac{n-1}{3}\rceil-1)$, $K(0, \lceil\frac{n-3}{2}\rceil)$, or $K$ has all black vertices (and white or gray edges).

\end{proof}

We only need to consider the $K\in \mathscr{K}(Forb(\widetilde{C_n}))$ with all black vertices such that $g_K(p)< \gamma_{Forb(\widetilde{C_n})}(p)$. Now, we establish some characterizations of such a $p$-core CRG $K$.

\begin{prop}\label{pro10}
Let $p\in [0, 1/2)$ and $K$ be a p-core CRG such that $K$ has only black vertices and white and gray edges. If $\widetilde{C_n} \not\mapsto K$ then $K$ has no gray cycle of length $l\in\{\left\lceil\frac{n-1}{2}\right\rceil, \ldots, n-1\}$.
\end{prop}

\begin{proof}
Suppose $K$ has some gray cycle of length $l\in\{\left\lceil\frac{n-1}{2}\right\rceil, \ldots, n-1\}$. Partition the vertices of $\widetilde{C_n}$ into $l$ parts so that one part is the triangle and each of the others parts is either a set of two consecutive vertices (an edge) or single vertex. Because of the structure of $\widetilde{C_n}$ and the fact that $\lceil\frac{n-1}{2}\rceil\leq l\leq n-1$, it is always possible to do so. This partition witnesses an embedding of $\widetilde{C_n}$ into the $l$-cycle of $K$ because we can map consecutive parts to consecutive vertices on the $l$-cycle. Since non-consecutive parts do not have edges between them and Proposition \ref{pro1}$(a)$ gives that the edges of $K$ are either white or gray, this map is an embedding that demonstrates $\widetilde{C_n}\mapsto K$, a contradiction.

\end{proof}

\begin{prop}\label{pro4}
Let $p\in\left[\frac{1}{\left\lceil\frac{n-1}{3}\right\rceil}, \frac{1}{2}\right)$, and $K$ be a $p$-core CRG with all black vertices such that $g_K(p)< \gamma_{Forb(\widetilde{C_n})(p)}$. Then:

$(a)$ for every $v\in V(K)$, $deg_G(v)\geq\left\lceil\frac{n-1}{3}\right\rceil$, and

$(b)$ for every $v,w\in V(K)$, $deg_G(v,w)\geq 1$.
\end{prop}

\begin{proof}
$(a)$ Let $v, w \in V(K)$. By using Proposition {\ref{pro2}}$(a)$,
\begin{align*}
deg_G(v)&\geq\left\lceil\frac{d_G(v)}{\max\{\mathbf{x}(w)\}}\right\rceil\geq \frac{\frac{p-g_K(p)}{p}+\frac{1-2p}{p}\mathbf{x}(v)}{\frac{g_K(p)}{1-p}}\\
&\geq \frac{(p-g_K(p))(1-p)}{pg_K(p)}=\frac{1-p}{g_K(p)}-\frac{1-p}{p}\\
&>\frac{(1-p)+(\lceil\frac{n-1}{3}\rceil-1)p}{p}-\frac{1-p}{p}\\
&=\left\lceil\frac{n-1}{3}\right\rceil-1.
\end{align*}

$(b)$ By the inclusion-exclusion principle, $d_G(v)+d_G(w)-d_G(v,w)\leq 1$, and by using Proposition \ref{pro2}$(a)$, we have
$d_G(v,w)\geq 2\frac{p-g_K(p)}{p}+\frac{1-2p}{p}(\mathbf{x}(v)+\mathbf{x}(w))-1\geq\frac{p-g_K(p)}{p}\geq \frac{p-2g_K(p)}{p}$ and for all $u \in V(K)$, $\mathbf{x}(u)\leq g_K(p)/(1-p)$. Therefore,
\begin{align*}
deg_G(v,w)&\geq\left\lceil\frac{d_G(v,w)}{\max\{\mathbf{x}(u)\}}\right\rceil\geq \left\lceil\frac{\frac{p-2g_K(p)}{p}}{\frac{g_K(p)}{1-p}}\right\rceil=\frac{1-p}{g_K(p)}-\frac{2(1-p)}{p}\\
&>\frac{(1-p)+(\lceil\frac{n-1}{3}\rceil-1)p}{p}-\frac{2(1-p)}{p}\\
&=\left\lceil\frac{n-1}{3}\right\rceil-\frac{1}{p}.
\end{align*}

Since $p\geq \frac{1}{\lceil\frac{n-1}{3}\rceil}$, we have $deg_G(v,w)\geq 1$.

\end{proof}

We consider the value of $ed_{Forb(\widetilde{C_n})}(p)$ from the perspective of the gray subgraphs of CRGs $K$. Let $F$ be a graph such that $V(F)=V(K)$ and $E(F)=EG(K)$, where $K\in \mathscr{K}(Forb(\widetilde{C_n}))$ is a $p$-core CRG with all black vertices such that $g_K(p)<\gamma_{Forb(\widetilde{C_n})}(p)$.
By Proposition \ref{pro4}, $F$ is a connected graph and each pair of vertices has at least one common neighbor.

\begin{prop}\label{pro11}
Let $n\geq9$ and $F$ be a graph with no cycle with length in $\left\{\left\lceil\frac{n-1}{2}\right\rceil, \ldots, n-1\right\}$ and every pair of vertices having at least one common neighbor. Then $F$ has no cycle of with length greater than $\left\lceil\frac{n-1}{2}\right\rceil-1$.
\end{prop}

\begin{proof}
Let $v_1 \ldots v_\ell v_1$ be a shortest cycle in $F$ among all those with length greater than $n-1$. Consider the path $v_1 \ldots v_{\lceil\frac{n-1}{2}\rceil-1}$ on the cycle $v_1 \ldots v_\ell v_1$.

Assume $v_i$ is a common neighbor of $v_1$ and $v_{\lceil\frac{n-1}{2}\rceil-1}$, then either $v_1v_iv_{i+1}\ldots v_\ell v_1$ or $v_1 \ldots v_iv_{\lceil\frac{n-1}{2}\rceil-1}\ldots v_\ell v_1$ has length less than $\ell$. Without loss of generality, we assume  $v_1v_iv_{i+1}\ldots v_\ell v_1$ has length less than $\ell$, which implies
$$\left\lceil\frac{n-1}{2}\right\rceil-1\geq\ell-i+2\geq\ell-\left(\left\lceil\frac{n-1}{2}\right\rceil-2\right)+2\geq n-\left\lceil\frac{n-1}{2}\right\rceil+4.$$
Thus,
$$2\left\lceil\frac{n-1}{2}\right\rceil-1-n-4\geq 0,$$
a contradiction, since
$2\left\lceil\frac{n-1}{2}\right\rceil-1-n-4< 2(\frac{n-1}{2}+1)-1-n-4< 0$.

Therefore, $F$ has no cycle of with length greater than $\left\lceil\frac{n-1}{2}\right\rceil-1$.

\end{proof}

Then, we consider the maximum-length path in the graph $F$. If this path forms a cycle, then Proposition {\ref{pro5}} gives that $F$ must be Hamiltonian. By Proposition {\ref{pro11}}, $|V(K)|\leq \lceil\frac{n-1}{2}\rceil-1$ and $g_K(p)\geq \frac{1-p}{\lceil\frac{n-1}{2}\rceil-1}$, a contradiction. Thus, no maximum-length path in $F$ forms a cycle. By Proposition \ref{pro11}, $F$ has no cycle of with length greater than $\lceil\frac{n-1}{2}\rceil-1$. And, by Proposition \ref{pro4}, every vertex in $F$ has degree at least $\left\lceil\frac{n-1}{3}\right\rceil\geq2$ and every pair of vertices has at least one common neighbor.

Let $v_1 \ldots v_\ell$ be a maximum-length path in $F$ such that the sum $\mathbf{x}(v_1) + \mathbf{x}(v_\ell)$ is largest among all such paths. Then by Proposition {\ref{pro6}}, we have $v_1$ and $v_\ell$ have a unique common neighbor $v_c$ and $N(v_1)\subseteq\{v_2, \ldots, v_c\}$. Let $v_1$ have $d$ neighbors in $F$.
Since $v_1$ cannot have neighbors outside of this path, $d_G(v_1)\leq \mathbf{x}(v_2)+\cdots+\mathbf{x}(v_c)$. And if $v_i\in\{v_1, \ldots, v_{c-1}\}$ is a predecessor of a neighbor of $v_1$ in $F$, then it is an endpoint of a path containing the same $\ell$ vertices, namely $v_iv_{i-1}\ldots v_1v_{i+1}v_{i+2}\ldots v_c\ldots v_\ell$. Hence all $d$ predecessors of gray neighbors of $v_1$ (including $v_1$ itself) have weight at most $\mathbf{x}(v_1)$. By Proposition {\ref{pro2}}, $\frac{p-g_K(p)}{p}+\frac{1-p}{p}\mathbf{x}(v_1)=\mathbf{x}(v_1)+d_G(v_1)
\leq \mathbf{x}(v_1)+\cdots+\mathbf{x}(v_c)
\leq d\mathbf{x}(v_1)+(c-d)\frac{g}{1-p}$,
which implies
\begin{equation*}\label{e6}
g_K(p)\left(\frac{c-d}{1-p}+\frac{1}{p}\right)\geq 1-\mathbf{x}(v_1)\left(d-\frac{1-p}{p}\right).
\end{equation*}
By Propositions {\ref{pro10} and {\ref{pro4}},} we have $c\leq \lceil\frac{n-1}{2}\rceil-1$ and $d>\lceil\frac{n-1}{3}\rceil-1$. So when $p\geq\lceil\frac{n-1}{3}\rceil^{-1}$, by Proposition {\ref{pro2}$(a)$}, we have $\mathbf{x}(v)\leq g_K(p)/(1-p)$, hence
$$g_K(p)\geq\frac{1-p}{c}\geq\frac{1-p}{\lceil\frac{n-1}{2}\rceil-1}\geq\gamma_{\mathscr{H}}(p),$$ a contradiction. So $ed_\mathscr{H}(p)=\gamma_\mathscr{H}(p)$ for all $p\in \left[\frac{1}{\lceil\frac{n-1}{3}\rceil}, \frac{1}{2}\right)$.

Finally, $ed_{\mathscr{H}}(p)=\gamma_{\mathscr{H}}(p)=\frac{p}{2}$ for $p=\frac{1}{\left\lceil\frac{n-1}{3}\right\rceil}$, and $ed_{\mathscr{H}}(p)=\gamma_{\mathscr{H}}(p)=\frac{p}{2}$ for $p=0$. Then, since the function $\gamma_{\mathscr{H}}(p)$ is linear over this interval and $ed_{\mathscr{H}}(p)$ is continuous and concave down, we have $ed_{\mathscr{H}}(p)=\gamma_{\mathscr{H}}(p)$ for $p\in \left[0,\frac{1}{\left\lceil\frac{n-1}{3}\right\rceil}\right]$. Hence the two functions are equal for all $p\in[0, 1]$.

\section{Proof of Theorem \ref{theorem2}}\label{Sec4}
Similarly as Section \ref{Sec5}, but it also involves some crucial differences.
We first prove the following lemma.

\begin{lemma}\label{theorem3}
Let $\mathscr{H}=Forb(P_n)$ where $P_n$ denotes the path on $n\geq 3$ vertices.
$$\gamma_{\mathscr{H}}(p)= \min\left\{\frac{p(1-p)}{1+(\lceil\frac{n-1}{3}\rceil-2)p},\ \frac{1-p}{\lceil\frac{n}{2}\rceil-1}\right\},\ \text{for}\ p\in[0, 1].$$
Furthermore, if there is a $p$-core CRG $K\in\mathscr{K}(\mathscr{H})$ such that $g_K(p)<\gamma_\mathscr{H}(p)$ for any $p\in(0, 1)$, then $p< \frac{1}{2}$ and $K$ has all black vertices.
\end{lemma}

\begin{proof}
The extreme points of the clique spectrum of $Forb(P_n)$ are $(1, \lceil\frac{n-1}{3}\rceil-1)$ and $(0, \lceil\frac{n}{2}\rceil-1)$, which establishes the value of $\gamma_\mathscr{H}(p)$.

Since $ed_\mathscr{H}(1/2)=\gamma_\mathscr{H}(1/2)$ for any hereditary property and $\gamma_\mathscr{H}(1)=0$, we may use continuity and concavity to conclude that $ed_\mathscr{H}(p)=\frac{1-p}{\lceil\frac{n}{2}\rceil-1}$ for $p\in [1/2, 1]$.

Now, let $p\in[0, 1/2)$ and $K$ be a $p$-core CRG such that $P_n\not\mapsto K$. If $K$ has only white vertices, then $K\approx K(1,0)$ and $g_K(p)=p>\gamma_{\mathscr{H}}(p)$.
If $K$ has both white and black vertices, then it has at most one white vertex since $P_n \mapsto K(2,1)$. Furthermore, it can have at most $\lceil\frac{n-1}{3}\rceil-1$ black vertices. To see this, denote the vertices of $P_n$ by $\{0,1,\ldots,n-1\}$ where $0\sim1\sim2\sim\cdots\sim n-1$. Let $S$ consist of the members of $\{0,1,\ldots,n-1\}$ that are divisible by 3. The graph $P_n-S$ has $\lceil\frac{n-1}{3}\rceil$ connected components, each of which are cliques of size 1 or 2. Thus, regardless of whether the edges are white or gray, there are at most $\lceil\frac{n-1}{3}\rceil-1$ black vertices in $K$ and $g_K(p)\geq\frac{p(1-p)}{1+(\lceil\frac{n-1}{3}\rceil-2)p}$, with equality if and only if $K\approx K(1, \lceil\frac{n-1}{3}\rceil-1)$.

Summarizing, if $p\in[0,1/2)$ and $g_K(p)=ed_\mathscr{H}(p)$, then $K$ is either $K(1, \lceil\frac{n-1}{3}\rceil-1)$, $K(0, \lceil\frac{n}{2}\rceil-1)$ or $K$ has all black vertices (and white or gray edges).

\end{proof}

When $n < 5$, $\gamma_{\mathscr{H}}(p)=\min\{p, 1-p\}$. This observation plus continuity and concavity give that $ed_{\mathscr{H}}(p)=\gamma_{\mathscr{H}}(p)$ for all $p\in[0,1]$. From now on, we assume $n\geq5$.

We only need to consider the $K\in \mathscr{K}(Forb(P_n))$ with all black vertices such that $g_K(p)< \gamma_{Forb(P_n)}(p)$. Now, we establish some characterizations of such a $p$-core CRG $K$.

\begin{prop}\label{pro13}
Let $p\in\left[\frac{1}{\left\lceil\frac{n-1}{3}\right\rceil}, \frac{1}{2}\right)$, and $K$ be a $p$-core CRG with all black vertices such that $g_K(p)<\gamma_{Forb(P_n)}(p)$. Then:

$(a)$ for every $v\in V(K)$, $deg_G(v)\geq\left\lceil\frac{n-1}{3}\right\rceil$, and

$(b)$ for every $v,w\in V(K)$, $deg_G(v,w)\geq 1$.
\end{prop}

\begin{proof}
$(a)$ Let $v,w \in V(K)$. By using Proposition {\ref{pro2}}$(a)$,
\begin{align*}
deg_G(v)&\geq\left\lceil\frac{d_G(v)}{\max\{\mathbf{x}(w)\}}\right\rceil\geq \frac{\frac{p-g_K(p)}{p}+\frac{1-2p}{p}\mathbf{x}(v)}{\frac{g_K(p)}{1-p}}\\
&\geq \frac{(p-g_K(p))(1-p)}{pg_K(p)}=\frac{1-p}{g_K(p)}-\frac{1-p}{p}\\
&>\frac{(1-p)+(\lceil\frac{n-1}{3}\rceil-1)p}{p}-\frac{1-p}{p}\\
&=\left\lceil\frac{n-1}{3}\right\rceil-1.
\end{align*}

$(b)$ By the inclusion-exclusion principle, $d_G(v)+d_G(w)-d_G(v,w)\leq 1$, and by using Proposition \ref{pro2}$(a)$, we have
$d_G(v,w)\geq 2\frac{p-g_K(p)}{p}+\frac{1-2p}{p}(\mathbf{x}(v)+\mathbf{x}(w))-1 \geq\frac{p-g_K(p)}{p}\geq \frac{p-2g_K(p)}{p}$ and for all $u \in V(K)$, $\mathbf{x}(u)\leq g_K(p)/(1-p)$. Therefore,
\begin{align*}
deg_G(v,w)&\geq\left\lceil\frac{d_G(v,w)}{\max\{\mathbf{x}(u)\}}\right\rceil\geq \left\lceil\frac{\frac{p-2g_K(p)}{p}}{\frac{g_K(p)}{1-p}}\right\rceil=\frac{1-p}{g_K(p)}-\frac{2(1-p)}{p}\\
&>\frac{(1-p)+(\lceil\frac{n-1}{3}\rceil-1)p}{p}-\frac{2(1-p)}{p}\\
&=\left\lceil\frac{n-1}{3}\right\rceil-\frac{1}{p}.
\end{align*}
Since $p\geq \frac{1}{\lceil\frac{n-1}{3}\rceil}$, we have $deg_G(v,w)\geq 1$.

\end{proof}

\begin{prop}\label{pro12}
Let $p\in [0, 1/2)$ and $K$ be a $p$-core CRG such that $K$ has only black vertices and white and gray edges. If $P_n \not\mapsto K$ then $K$ has no gray path with length greater than $\lceil\frac{n}{2}\rceil-1$.
\end{prop}

\begin{proof}
Suppose $K$ has some gray path of length $l>\lceil\frac{n}{2}\rceil-1$. Partition the vertices of $P_n$ into $l$ parts so that each of parts is either a set of two consecutive vertices (an edge) or single vertex. Because of the structure of $P_n$ and the fact that $l>\lceil\frac{n}{2}\rceil-1$, it is always possible to do so. This partition witnesses an embedding of $P_n$ into $l$-path of $K$ because we can map consecutive parts to consecutive vertices on the $l$-path. Since non-consecutive parts do not have edges between them and Proposition \ref{pro1}$(a)$ gives that the edges of $K$ are either white or gray, this map is an embedding that demonstrates $P_n\mapsto K$, a contradiction.

\end{proof}

We consider the value of $ed_{Forb(P_n)}(p)$ from the perspective of the gray subgraphs of CRGs $K$. Let $F$ be a graph, $V(F)=V(K)$, $E(F)=EG(K)$ where $K\in \mathscr{K}(Forb(P_n))$ is a $p$-core CRG with all black vertices such that $g_K(p)<\gamma_{Forb(P_n)}(p)$. By Proposition \ref{pro13}, we can obtain $F$ is a connected graph.

Suppose a maximum-length path forms a cycle in the graph $F$. Then Proposition {\ref{pro5}} implies that $F$ must be a Hamiltonian. By Proposition {\ref{pro12}, $|V(K)|\leq\left\lceil\frac{n}{2}\right\rceil-1$ and $g_K(p)\geq\frac{1-p}{\lceil\frac{n}{2}\rceil-1}$, a contradiction, and so we may assume that no maximum-length path in $F$ forms a cycle.
By Proposition {\ref{pro12}}, $F$ has no path with length greater than $\lceil\frac{n}{2}\rceil-1$, so $F$ has no cycle with length greater than $\lceil\frac{n}{2}\rceil-1$. And, by Proposition \ref{pro13}, every vertex in $F$ has degree at least $\left\lceil\frac{n-1}{3}\right\rceil\geq2$ and every pair of vertices has at least one common neighbor.

Let $v_1 \ldots v_\ell$ be such a maximum-length path in $K$ such that the sum $\mathbf{x}(v_1) + \mathbf{x}(v_\ell)$ is the largest among all such paths. By Proposition {\ref{pro6}}, $v_1$ and $v_\ell$ have a unique common neighbor $v_c$ and $N(v_1)\subseteq\{v_2, \ldots, v_c\}$. Let $v_1$ have $d$ neighbors in $F$. Since $v_1$ cannot have neighbors outside of this path, the sum of the weights of the neighbors of $v_1$ satisfies $d_G(v_1)\leq \mathbf{x}(v_2)+\cdots+\mathbf{x}(v_c)$ in $K$. And if $v_i\in\{v_1, \ldots, v_{c-1}\}$ is a predecessor of a neighbor of $v_1$, then it is an endpoint of a path containing the same $\ell$ vertices, namely $v_iv_{i-1}\ldots v_1v_{i+1}v_{i+2}\ldots v_c\ldots v_\ell$. Hence all $d$ predecessors of gray neighbors of $v_1$ (including $v_1$ itself) have weight at most $\mathbf{x}(v_1)$. By Proposition {\ref{pro2}}, $\frac{p-g_K(p)}{p}+\frac{1-p}{p}\mathbf{x}(v_1)=\mathbf{x}(v_1)+d_G(v_1)
\leq \mathbf{x}(v_1)+\cdots+\mathbf{x}(v_c)
\leq d\mathbf{x}(v_1)+(c-d)\frac{g_K(p)}{1-p}$,
which implies
$$g_K(p)\left(\frac{c-d}{1-p}+\frac{1}{p}\right)\geq 1-\mathbf{x}(v_1)\left(d-\frac{1-p}{p}\right).$$

By Propositions {\ref{pro13}} and {\ref{pro12}}, we have $c\leq \lceil\frac{n}{2}\rceil-1$ and $d\geq\lceil\frac{n-1}{3}\rceil$. And, when $\lceil\frac{n-1}{3}\rceil^{-1}\leq p\leq \frac{1}{2}$, by Proposition {\ref{pro2}(a)}, we have $\mathbf{x}(v)\leq g_K(p)/(1-p)$, hence
$$g_K(p)\geq\frac{1-p}{c}\geq\frac{1-p}{\lceil\frac{n}{2}\rceil-1}\geq\gamma_{\mathscr{H}}(p),$$
a contradiction.

So we can get $ed_{\mathscr{H}}(p)=\gamma_{\mathscr{H}}(p)$ for $p \in\left[\frac{1}{\lceil\frac{n-1}{3}\rceil}, 1\right]$. The proof is thus complete.

\section*{Acknowledgements}
The authors would like to thank the two referees for their careful reading of the manuscript and for their constructive suggestions provided in the reports, in particular, for pointing out the error in the proof. Yongtang Shi was partially supported by National Natural Science Foundation of China and Natural Science Foundation of Tianjin (No.~17JCQNJC00300).

\end{document}